\documentclass[sn-mathphys]{sn-jnl}

\usepackage{epstopdf}
\usepackage{lmodern}
\usepackage{graphicx}

\theoremstyle{thmstyleone}%
\newtheorem{theorem}{Theorem}
\newtheorem{proposition}[theorem]{Proposition}
\newtheorem{lem}[theorem]{Lemma}
\newtheorem{cor}[theorem]{Corollary}
\newtheorem{definition}{Definition}

\theoremstyle{thmstyletwo}%
\newtheorem{remark}{Remark}%

\begin{document}

\title[Splitting Gibbs Measures for Triple Mixed-Spin Ising Model]{Splitting Gibbs Measures for a Periodic Triple Mixed-Spin Ising Model on a Cayley Tree}


\author[1]{\fnm{Farrukh} \sur{Mukhamedov}}\email{far75m@yandex.ru; farrukh.m@uaeu.ac.ae}

\author[2]{\fnm{Muzaffar} \sur{Rahmatullaev}}\email{mrahmatullaev@rambler.ru}
\equalcont{These authors contributed equally to this work.}

\author[3]{\fnm{Obid} \sur{Karshiboev}}\email{okarshiboevsher@mail.ru}
\equalcont{These authors contributed equally to this work.}

\affil[1]{\orgdiv{Department of Mathematical Sciences}, \orgname{College of Science}, \orgaddress{\street{Street}, \city{Al Ain Abu Dhabi}, \postcode{15551}, \state{UAE}, \country{The United Arab Emirates University}}}

\affil[2]{\orgdiv{V.I. Romanovskiy Institute of Mathematics}, \orgname{Uzbekistan Academy of Sciences}, \orgaddress{\street{University str, 4-b}, \postcode{100174}, \state{Tashkent}, \country{Uzbekistan}}}

\affil*[3]{\orgdiv{Mathematics and Information Technology}, \orgname{Oriental University, Samarkand}, \orgaddress{\street{Spitamen shoh str. 139}, \city{Samarkand}, \postcode{140102}, \state{Samarkand}, \country{Uzbekistan}}}


\abstract{
We consider an Ising model on the Cayley tree $\Gamma_k$ of arbitrary order
$k\ge1$ with three spin species of values $(\tfrac12,1,\tfrac32)$ distributed
deterministically with period three along the generations. Within the framework
of splitting Gibbs measures, we derive the exact boundary-law compatibility
equations and characterize translation-invariant splitting Gibbs measures
(TISGMs) via a finite system of algebraic relations. In the ferromagnetic
regime $J>0$, writing $\theta=\exp(\beta J/2)$, we further reduce the
translation-invariant problem to a one-dimensional scalar fixed-point equation
$x=f(x,\theta,k)$ for a rational map $f$. We show that $f$ is strictly increasing and
obtain an explicit sufficient condition for phase coexistence: if
$s_k(\theta)=f'(1,\theta,k)-1>0$, then $x=f(x,\theta,k)$ admits at least three
distinct positive solutions, yielding at least three distinct TISGMs and hence
a phase transition driven by the periodic inhomogeneity of the spin structure.

For the binary tree $k=2$ we exploit attractiveness to construct plus and minus
Gibbs measures as weak limits with extremal boundary conditions, prove that
they are TISGMs corresponding to the minimal and maximal fixed points of
$f(\cdot,\theta,2)$, and show that they are the minimal and maximal Gibbs
measures in the natural stochastic order. Finally, we construct the
tree-indexed Markov chain associated with a TISGM and apply the
Kesten--Stigum criterion to the disordered TISGM, identifying nonempty
parameter regions where this measure is non-extremal and reconstruction occurs.
}




\maketitle

\section{Introduction}\label{sec0}

The Ising model remains one of the most fundamental frameworks in statistical mechanics for understanding cooperative phenomena and phase transitions in interacting systems. Since its introduction in the early twentieth century, it has inspired a vast body of work ranging from exact solutions on regular lattices to rigorous probabilistic approaches on tree-like structures such as Cayley trees and Bethe lattices \cite{Ising,G,Ro}.

In recent decades, \emph{mixed-spin Ising models}, in which different lattice sites host spins of unequal magnitude, have attracted sustained interest. These models provide more realistic descriptions of ferrimagnetic materials and molecular magnets and are known to exhibit richer critical behavior than their single-spin counterparts \cite{Ekiz2009,Jiang2003,Bouda,QMA25}. Numerous analytical and numerical techniques have been applied to mixed-spin systems, including exact recursion relations, effective-field and mean-field theories, and Monte Carlo simulations \cite{Kaneyoshi1993,Tucker2000,Bobak1998,Albayrak2010}.

A further level of complexity arises in \emph{ternary or triple mixed-spin Ising models}, where three distinct spin values coexist. Such models have been proposed as idealized descriptions of Prussian blue analogs and other molecular-based magnetic materials in which ferro- and ferrimagnetic interactions compete \cite{Bobak2002,Ohkoshi1997,Mallah1993}. Several triple mixed-spin models have been investigated on the Bethe lattice, revealing compensation points, magnetization reversals, and multicritical behavior \cite{Canko2008,Albayrak2011,Albayrak2020,Deviren2009,Deviren2010,KMU24}. However, most of these studies rely on approximate or numerical approaches and focus on translationally homogeneous environments.

From a rigorous mathematical perspective, Ising-type models on Cayley trees display behavior that sharply contrasts with that on regular lattices. In particular, the absence of loops allows for an exact characterization of Gibbs measures via compatibility conditions and recursive equations. While the classical Ising model on a Cayley tree exhibits phase transitions only in the ferromagnetic regime, mixed-spin models may admit multiple Gibbs measures even in antiferromagnetic settings \cite{akin2018gibbs,akin20231,akin20232,akin2024,MRE23}. Moreover, beyond the mere existence of multiple Gibbs measures, the \emph{extremality or non-extremality} of these measures plays a crucial role in understanding reconstruction problems, information flow on trees, and the stability of disordered phases \cite{Mossel,Martinelli,M22}.

In this paper, we study a \emph{triple mixed-spin Ising model with spin values $(1/2,1,3/2)$} defined on a semi-infinite Cayley tree of arbitrary order $k$. Unlike previously studied ternary mixed-spin models, the spins in our model are arranged \emph{periodically with period three along the generations of the tree}, rather than being uniformly distributed or randomly assigned. This deterministic but inhomogeneous structure leads to a fundamentally different system of recursive equations and allows for new phase-transition phenomena that do not arise in homogeneous mixed-spin settings.

Using the framework of \emph{splitting Gibbs measures}\cite{HK21,Ro}, we derive the exact compatibility equations associated with the model and analyze their translation-invariant solutions. We establish explicit sufficient conditions for the existence of multiple translation-invariant splitting Gibbs measures, thereby identifying regions of phase transitions in the parameter space.
The main result is formulated as follows (see section 4):

\begin{theorem}\label{thm:intro-phase}
Let $k\ge2$ and $J>0$, and put $\theta=\exp(\beta J/2)$.
Consider the triple mixed-spin Ising model with spin values
$\bigl(\tfrac12,1,\tfrac32\bigr)$ on the Cayley tree $\Gamma_k$, where the
spin species $\Psi,\Phi,\Upsilon$ are arranged periodically with period three
along the generations.

Let $f(\cdot,\theta,k)$ be the scalar map defined in \eqref{eq14} and set
\[
  s_k(\theta)\;=\;f'(1,\theta,k)-1\,.
\]
If $s_k(\theta)>0$, then the fixed-point equation
\[
  x \;=\; f(x,\theta,k)
\]
admits at least three distinct positive solutions. Consequently, for every such
pair $(\theta,k)$ the model admits at least three distinct translation-invariant
splitting Gibbs measures.
\end{theorem}

This instability criterion $s_k(\theta)>0$ is in the spirit of classical
linearization arguments around the symmetric fixed point used for mixed-spin
models on trees, but here the period-$3$ triple-spin structure leads to a
richer dependence on $(\theta,k)$ than in two-sublattice models. Numerical
evaluation of $s_k(\theta)$ shows that the region $\{\theta:s_k(\theta)>0\}$
widens as $k$ increases, indicating that the periodic inhomogeneity promotes
phase coexistence over a larger parameter range.

For the binary tree $k=2$ we exploit attractiveness to construct plus and minus
Gibbs measures as limits with all boundary spins fixed to their maximal or
minimal values. We prove that these measures $\mu^-$ and $\mu^+$ are TISGMs,
identify them with the extremal fixed points $x^-(\theta)$ and $x^+(\theta)$ of
the scalar recursion, and show that they are the minimal and maximal Gibbs
measures in the natural stochastic order (Proposition~12 and Lemma~13).
Under the assumption that $x=f(x,\theta,2)$ has exactly three positive fixed
points $x^-(\theta)<1<x^+(\theta)$, we obtain a full classification of
TISGMs and of ordered phases on the binary tree (Theorem~14) \cite{BLG,BRZ1}.

A second main direction of the paper is the extremality analysis of the
disordered phase $\mu_0$. To this end, we construct the tree-indexed Markov
chain induced by a TISGM and compute the effective $2\times2$ transition
matrix $H(\theta,k)$ for $\Psi$-spins along one period of the tree, obtained as
a product of three stochastic matrices associated with the levels
$\Psi\to\Phi\to\Upsilon\to\Psi$. Writing $\lambda_{0,k}$
for the second eigenvalue (in absolute value) of $H(\theta,k)$ and
\[
g_k(\theta):=k^3\lambda_{0,k}^2-1,
\]
the Kesten--Stigum bound \cite{KS} yields the following:

\begin{theorem}
Let $k\ge1$ and $J>0$, and write $\theta=\exp(\beta J/2)$. Denote by $\mu_0$
the translation-invariant splitting Gibbs measure corresponding to the fixed
point $x_0=1$ of the recursion~\eqref{67} (the disordered
phase). Let $H=H(\theta,k)$ be the effective $2\times2$ transition matrix from
$\Psi$ to $\Psi$ defined in~\eqref{H}, and let $\lambda_{0,k}$ be its
second eigenvalue in absolute value. If
\[
k^3\,\lambda_{0,k}^2>1,
\]
then the Gibbs measure $\mu_0$ is non-extremal.
\end{theorem}

Numerical analysis shows that the region $\{\theta:g_k(\theta)>0\}$ is nonempty
for every $k$ considered and tends to expand as $k$ increases. For $k=2$,
combining this with the classification of TISGMs yields a parameter regime
in which the disordered phase $\mu_0$ is non-extremal while the ordered
phases $\mu^-$ and $\mu^+$ remain extremal and exhaust all extremal TISGMs.

Finally, we derive thermodynamic quantities associated with a TISGM,
expressing the finite-volume free energy, the specific free energy, the
sublattice magnetizations on $\Gamma_k^0,\Gamma_k^1,\Gamma_k^2$, and the
magnetization and susceptibility along a typical ray, in terms of the
translation-invariant boundary law. These formulas make explicit how the
period-$3$ inhomogeneity influences macroscopic observables and provide a
bridge between the probabilistic description in terms of Gibbs measures and
the thermodynamic viewpoint based on bulk and ray observables.

The paper is organized as follows. In Section~2 we introduce the model,
the configuration spaces, and the splitting Gibbs measures, and we derive
the boundary-law compatibility system. Section~3 recalls TP2 kernels and
establishes attractiveness of the Gibbs specification, together with the
monotonicity properties of boundary laws. In Section~4 we specialize to
translation-invariant boundary laws and reduce the problem to the scalar
recursion $x=f(x,\theta,k)$. Section~5 is devoted to the dynamical analysis
of this recursion and to the existence of multiple TISGMs via the instability
criterion $s_k(\theta)>0$, with special attention to the binary case $k=2$
and the construction of plus/minus TISGMs. In Section~6 we study the
non-extremality of the disordered phase $\mu_0$ using the Kesten--Stigum
criterion and derive the function $g_k(\theta)$. 

\section{Preliminaries}\label{sec1}

Let $\Gamma^k = (V, L)$ denote a semi-infinite Cayley tree of order $k \geq 1$ with a designated root vertex $x^{(0)}$. Each vertex in the tree has exactly $k+1$ edges, except for the root $x^{(0)}$, which has exactly $k$ edges. Here, $V$ denotes the set of vertices and $L$ denotes the set of edges. Two vertices $x$ and $y$ are said to be \emph{nearest neighbors}, denoted by $l = \langle x, y \rangle$, if there exists an edge in $L$ connecting them.

A sequence of nearest-neighbor pairs $\langle x, x_1 \rangle, \ldots, \langle x_{d-1}, y \rangle$ is called a \emph{path} from vertex $x$ to vertex $y$. The \emph{distance} $d(x, y)$ between two vertices $x, y \in V$ is defined as the length of the shortest path connecting them.

For each non-negative integer $n$, define the following subsets of the Cayley tree:
\[
W_n = \{ x \in V \mid d(x, x^{(0)}) = n \}, \quad
V_n = \bigcup_{m=0}^{n} W_m, \quad
L_n = \{ \langle x, y \rangle \in L \mid x, y \in V_n \}.
\]
The \emph{set of direct successors} of a vertex $x \in W_n$ is defined as
\[
S(x) = \{ y \in W_{n+1} \mid d(x, y) = 1 \}.
\]

We define levels of the tree as follows:
\[
\Gamma_i^k = \{ x \in V \mid d(x^{(0)}, x) \equiv i \pmod{3} \},~~i=0,1,2.
\]

In this work, we consider spin state spaces given by $\Psi = \left\{ -\frac{1}{2},\frac{1}{2} \right\},$ $\Phi = \{-1, 0, 1\},$ $\Upsilon=\{-\frac{3}{2},-\frac{1}{2},\frac{1}{2},\frac{3}{2}\}.$ The associated configuration spaces are defined by
\[
\Omega_0 = \Psi^{\Gamma_0^k},~\Omega_1 = \Phi^{\Gamma_1^k},~\Omega_2 = \Upsilon^{\Gamma_2^k}.
\]
Additionally, we define the \emph{finite-volume configuration spaces} as
\[
\Omega_{0,n} = \Psi^{\Gamma_0^k\cap V_n},~\Omega_{1,n} = \Phi^{\Gamma_1^k\cap V_n},~\Omega_{2,n} = \Upsilon^{\Gamma_2^k\cap V_n}.
\]
The full configuration space of the model is given by
\[
\Xi = \Omega_0 \times \Omega_1\times \Omega_2.
\]
Elements of $\Omega_0$ are denoted by $\sigma(x)$ for $x \in \Gamma_0^k$, and elements of $\Omega_1$ are denoted by $s(x)$ for $x \in \Gamma_1^k$, elements of $\Omega_2$ are denoted by $\varphi(x)$ for $x \in \Gamma_2^k.$ For a configuration $\xi \in \Xi$, the associated spins are assigned to alternating generations of the tree (see Fig. \ref{fig20}). Spins are distributed periodically along the Cayley tree with period three. Vertices at levels congruent to $0 \pmod{3}$ carry spins from
$\Psi=\{-\tfrac12,\tfrac12\}$, vertices at levels congruent to $1 \pmod{3}$ carry spins from $\Phi=\{-1,0,1\}$, and vertices at levels congruent to $2 \pmod{3}$ carry spins from $\Upsilon=\{-\tfrac32,-\tfrac12,\tfrac12,\tfrac32\}$.
\begin{figure}
\centering
\includegraphics[width=10cm]{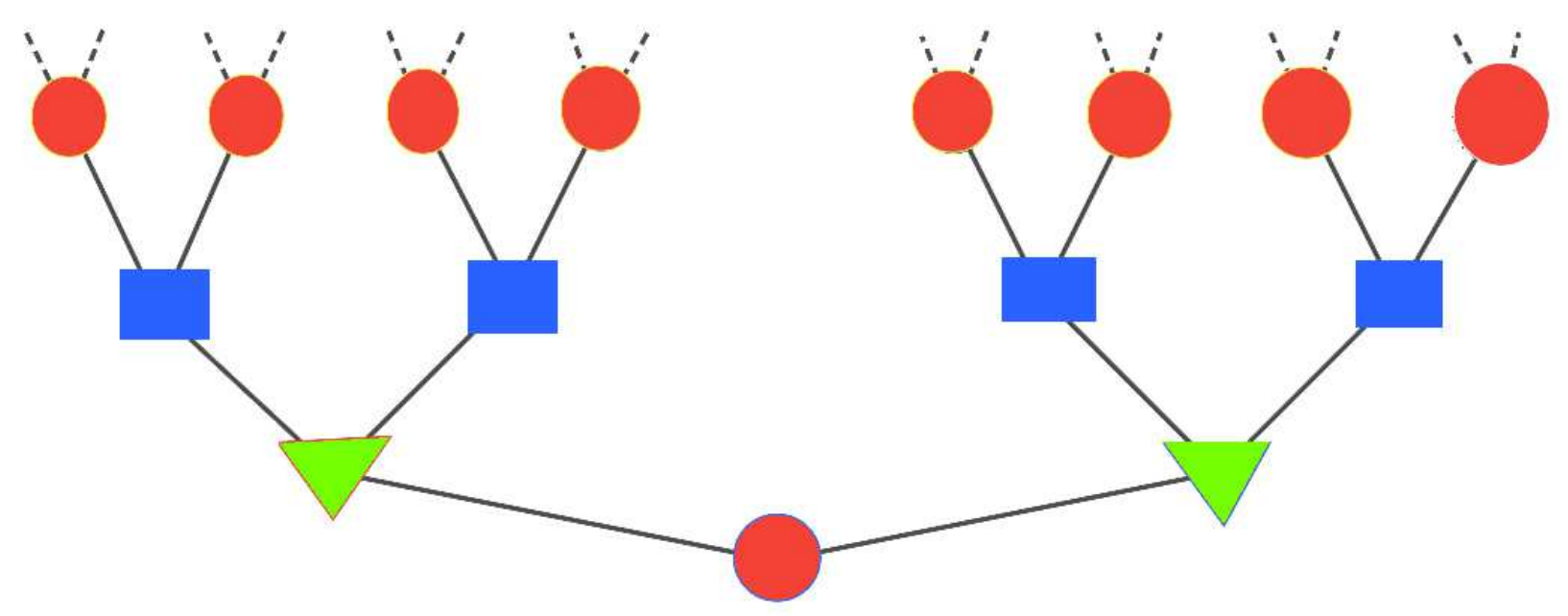}
\caption{Some generations of a second-order Cayley tree with three different spin values assigned to alternating levels.}\label{fig20}
\end{figure}

We define the Ising model with triple mixed spins $(\tfrac{1}{2}, 1, \tfrac{3}{2})$ governed by the Hamiltonian
\begin{equation} \label{eq1}
H(\xi) = -J \sum_{\langle x, y \rangle} \xi(x)\xi(y), \quad \xi \in \Xi.
\end{equation}
Let us denote $h = \{ h_x \}_{x \in \Gamma^k}$, where
\[
h_x =
\begin{cases}
    \tilde{h}_x, & x \in \Gamma_0^k, \\[2mm]
    \ddot{h}_x, & x \in \Gamma_1^k,\\[2mm]
    \widehat{h}_x, & x \in \Gamma_2^k,
\end{cases}
\]
with $\tilde{h}_x = (\tilde{h}_{-\frac{1}{2},x}, \tilde{h}_{\frac{1}{2},x}),$ $ \ddot{h}_x = (\ddot{h}_{-1,x}, \ddot{h}_{0,x}, \ddot{h}_{+1,x})$ and $ \widehat{h}(x) = (\widehat{h}_{-\frac{3}{2}}(x), \widehat{h}_{-\frac{1}{2}}(x), \widehat{h}_{\frac{1}{2}}(x), \widehat{h}_{\frac{3}{2}}(x))$.

For each $n \geq 1$, the finite-volume Gibbs measure $\mu_n^h$ is defined as
\begin{equation} \label{eq2}
\mu_n^h(\xi) = \frac{1}{Z_n} \exp\left\{ -\beta H_n(\xi) + \sum_{x \in W_n} h_{\xi(x)}(x) \right\}, \quad \xi \in \Xi_n := \Omega_{0,n} \times \Omega_{1,n}\times \Omega_{2,n},
\end{equation}
where $Z_n$ is the partition function, and $\beta>0$ is the inverse temperature.

The sequence of measures $\{ \mu_n^h \}$ is said to be compatible if, for all $n \geq 1$ and $\xi_{n-1} \in \Xi_{n-1}$, the following condition holds:
\begin{equation} \label{eq3}
\sum_{\omega \in \Xi^{W_n}} \mu_n^h(\xi_{n-1} \vee \omega) = \mu_{n-1}^h(\xi_{n-1}),
\end{equation}
where
\[
\Xi^{W_n} =
\begin{cases}
    \Psi^{W_n}, & \text{if } n \equiv 0 (\mathrm{mod}\,3), \\[2mm]
    \Phi^{W_n}, & \text{if } n \equiv 1 (\mathrm{mod}\,3), \\[2mm]
    \Upsilon^{W_n}, & \text{if } n \equiv 2 (\mathrm{mod}\,3).
\end{cases}
\]
Here, $\xi_{n-1} \vee \omega$ denotes the concatenation of the configurations. In this setting, there exists a unique probability measure $\mu$ on $\Xi$ such that for all $n$ and $\xi_n \in \Xi_n$,
\[
\mu\left( \{ \xi\mid_{V_n} = \xi_n \} \right) = \mu_n^h(\xi_n).
\]
Such a measure is called a \emph{splitting Gibbs measure} (SGM) associated with the model.

The following theorem is standard in the Cayley tree context, see, e.g., \cite{akin2018gibbs}:

\begin{theorem} \label{thm1}
The sequence of measures $\{\mu_n^h\},~n = 1, 2, \ldots$ defined by \eqref{eq2} is compatible if and only if, for all $x \in V$, the following system of equations holds:
\begin{equation}\label{67}
\left\{%
\begin{array}{llllll}
X_x=\prod\limits_{y \in S(x)}\frac{Y_y+\theta+\theta^2\,Z_y}{\theta^2\,Y_y+\theta+Z_y}, \\[4mm]
Y_x=\prod\limits_{y \in S(x)}\frac{\theta^6\,T_y+\theta^4+\theta^2\,U_y+N_y}{\theta^3(T_y+1+U_y+N_y)}, \\[4mm]
Z_x=\prod\limits_{y \in S(x)}\frac{T_y+\theta^2+\theta^4\,U_y+\theta^6\,N_y}{\theta^3(T_y+1+U_y+N_y)}, \\[4mm]
T_x=\prod\limits_{y \in S(x)}\frac{X_y+\theta^3}{\theta^2+\theta\,X_y}, \\[4mm]
U_x=\prod\limits_{y \in S(x)}\frac{\theta\,X_y+1}{X_y+\theta}, \\[4mm]
N_x=\prod\limits_{y \in S(x)}\frac{\theta^3\,X_y+1}{\theta^2+\theta\,X_y}.
\end{array}%
\right.
\end{equation}
where
\[\theta=\exp\Big(\frac{1}{2}\beta J\Big),\]
\[X_x=\exp(\tilde{h}_{\frac{1}{2},x}-\tilde{h}_{-\frac{1}{2},x}),~Y_x=\exp(\ddot{h}_{-1,x}-\ddot{h}_{0,x}),\]
\[Z_x=\exp(\ddot{h}_{1,x}-\ddot{h}_{0,x}),~T_x=\exp(\widehat{h}_{-\frac{3}{2},x}-\widehat{h}_{-\frac{1}{2},x}),\]
\[U_x=\exp(\widehat{h}_{\frac{1}{2},x}-\widehat{h}_{-\frac{1}{2},x}),~N_x=\exp(\widehat{h}_{\frac{3}{2},x}-\widehat{h}_{-\frac{1}{2},x}).\]

\end{theorem}


\section{Total positivity of order $2$ kernels, attractiveness of Gibbs specifications}

We recall that two-site Boltzmann factor along an edge $\langle x,y\rangle$ is defined by
\begin{equation}\label{eq:kernel}
K(s,t) := \exp\bigl(\beta J\, s t\bigr),\qquad s\in S_x,\ t\in S_y,
\end{equation}
where $S_x$ denotes the (finite, totally ordered) spin-set at $x$.
It is convenient to write
If $J>0$, then  $\theta>1$.

\begin{definition}\label{def:tp2}
Let $S,T$ be totally ordered finite sets.
A nonnegative kernel $K:S\times T\to (0,\infty)$ is called \emph{total positivity of order 2 (TP2)} if for all
$s_1<s_2$ in $S$ and $t_1<t_2$ in $T$ one has
\begin{equation}\label{eq:tp2}
K(s_1,t_1)\,K(s_2,t_2)\ \ge\ K(s_1,t_2)\,K(s_2,t_1).
\end{equation}
Equivalently: every $2\times 2$ minor of the matrix $(K(s,t))_{s\in S,t\in T}$ is nonnegative.
\end{definition}

\begin{lem}\label{lem:Ktp2}
Fix totally ordered finite sets $S,T\subset\mathbb{R}$ and consider the kernel $K(s,t)$
as in \eqref{eq:kernel}.
Then $K$ is TP2 if and only if $J\ge 0$.
Moreover, if $J>0$, then the inequality in \eqref{eq:tp2} is strict whenever
$s_1<s_2$ and $t_1<t_2$.
\end{lem}

\begin{proof}
Let $s_1<s_2$ and $t_1<t_2$.
Compute the cross-ratio:
\begin{eqnarray*}
\frac{K(s_1,t_1)\,K(s_2,t_2)}{K(s_1,t_2)\,K(s_2,t_1)}
&=&
\exp\!\Bigl(\beta J\bigl[s_1t_1+s_2t_2-s_1t_2-s_2t_1\bigr]\Bigr)\\[2mm]
&=&
\exp\!\Bigl(\beta J (s_2-s_1)(t_2-t_1)\Bigr).
\end{eqnarray*}
Since $(s_2-s_1)(t_2-t_1)>0$, the above ratio is $\ge 1$ if and only if $J\ge 0$.
This is exactly \eqref{eq:tp2}; strictness holds when $J>0$.
\end{proof}

\begin{remark}
We point out that TP2 is equivalent to \emph{log-supermodularity}:
\[
\log K(s_1,t_1)+\log K(s_2,t_2)\ \ge\ \log K(s_1,t_2)+\log K(s_2,t_1),
\]
which is often the form used in FKG/Holley arguments.
\end{remark}

Now, we use the coordinate-wise partial order on configurations.
For a finite vertex set $\Lambda\subset \Gamma_k$, write $\Omega_\Lambda:=\prod_{x\in\Lambda} S_x$,
and for $\sigma,\tau\in\Omega_\Lambda$ define
\[
\sigma\le\tau \quad\Longleftrightarrow\quad \sigma(x)\le \tau(x)\ \text{for all }x\in\Lambda,
\]
and pointwise joins/meets $$(\sigma\vee\tau)(x)=\max\{\sigma(x),\tau(x)\}, \quad
(\sigma\wedge\tau)(x)=\min\{\sigma(x),\tau(x)\}.$$

We recall that a Gibbs specification $\gamma_{\Lambda}(\cdot \mid \eta)$ is \emph{attractive}
if for every finite $\Lambda$ and every pair of boundary conditions
$\eta\le \xi$ on $\Lambda^c$, one has stochastic domination:
for every bounded \emph{increasing} function $F:\Omega_\Lambda\to\mathbb{R}$,
\begin{equation}\label{eq:attr}
\mathbb{E}_{\gamma_\Lambda(\cdot \mid\eta)}[F]\ \le\ \mathbb{E}_{\gamma_\Lambda(\cdot\mid\xi)}[F].
\end{equation}

\begin{theorem}\label{thm:tp2-attractive}
Consider a finite graph $G=(V,E)$ with totally ordered finite spin sets $\{S_x\}_{x\in V}$.
Assume the (unnormalized) Gibbs weight has the pair-product form
\begin{equation}\label{eq:weight}
w(\sigma)\ :=\ \prod_{\{x,y\}\in E} K_{xy}(\sigma(x),\sigma(y))\ \cdot\ \prod_{x\in V} a_x(\sigma(x)),
\end{equation}
where $a_x>0$ and each edge kernel $K_{xy}:S_x\times S_y\to (0,\infty)$ is TP2.
Then the associated Gibbs specification is attractive.
In particular, for the Ising-type kernel $K_{xy}(s,t)=\exp(\beta J st)$,
the specification is attractive whenever $J>0$ (equivalently $\theta>1$).
\end{theorem}

\begin{proof}
For each $x\in V$, the multiset identity $\{(\sigma\wedge\tau)(x),(\sigma\vee\tau)(x)\}=\{\sigma(x),\tau(x)\}$ implies
\[
a_x\big((\sigma\wedge\tau)(x)\big)\,a_x\big((\sigma\vee\tau)(x)\big)=a_x\big(\sigma(x)\big)\,a_x\big(\tau(x)\big).
\]
Thus it suffices to prove the lattice condition edge-by-edge.

Fix an edge $\{x,y\}\in E$ and set
\[
s_1:=(\sigma\wedge \tau)(x),\quad s_2:=(\sigma\vee\tau)(x),\quad
t_1:=(\sigma\wedge \tau)(y),\quad t_2:=(\sigma\vee\tau)(y).
\]
Then $s_1\le s_2$ in $S_x$ and $t_1\le t_2$ in $S_y$, and
\[
K_{xy}\big((\sigma\wedge\tau)(x),(\sigma\wedge\tau)(y)\big)\,
K_{xy}\big((\sigma\vee\tau)(x),(\sigma\vee\tau)(y)\big)
=K_{xy}(s_1,t_1)\,K_{xy}(s_2,t_2).
\]
If $\sigma(x)\le\tau(x)$ and $\sigma(y)\le\tau(y)$ (or $\sigma(x)\ge\tau(x)$ and $\sigma(y)\ge\tau(y)$), then
$$(\sigma\wedge\tau)(x)=\sigma(x), \quad (\sigma\wedge\tau)(y)=\sigma(y)$$ and
$$(\sigma\vee\tau)(x)=\tau(x),\quad (\sigma\vee\tau)(y)=\tau(y),$$ hence
\[
K_{xy}(s_1,t_1)\,K_{xy}(s_2,t_2)=K_{xy}(\sigma(x),\sigma(y))\,K_{xy}(\tau(x),\tau(y)).
\]
Otherwise the local orderings at $x$ and $y$ are opposite, and one checks that
\[
K_{xy}(\sigma(x),\sigma(y))\,K_{xy}(\tau(x),\tau(y))=K_{xy}(s_1,t_2)\,K_{xy}(s_2,t_1).
\]
By the TP2 property of $K_{xy}$,
\[
K_{xy}(s_1,t_1)\,K_{xy}(s_2,t_2)\ \ge\ K_{xy}(s_1,t_2)\,K_{xy}(s_2,t_1),
\]
so in all cases
\[
K_{xy}\big((\sigma\wedge\tau)(x),(\sigma\wedge\tau)(y)\big)\,
K_{xy}\big((\sigma\vee\tau)(x),(\sigma\vee\tau)(y)\big)
\ \ge\
K_{xy}(\sigma(x),\sigma(y))\,K_{xy}(\tau(x),\tau(y)).
\]
Multiplying over all edges and using the single-site identity gives
\[
w(\sigma\wedge\tau)\,w(\sigma\vee\tau)\ \ge\ w(\sigma)\,w(\tau),\qquad \forall \sigma,\tau\in\Omega_V.
\]

Now fix $\Lambda\Subset V$ and boundary conditions $\eta\le\xi$ on $\Lambda^c$. Define
$$w_\Lambda^\eta(\sigma_\Lambda):=w(\sigma_\Lambda\eta_{\Lambda^c})\quad w_\Lambda^\xi(\sigma_\Lambda):=w(\sigma_\Lambda\xi_{\Lambda^c}).$$
For $\sigma_\Lambda,\tau_\Lambda\in\Omega_\Lambda$, set
$\tilde\sigma:=\sigma_\Lambda\eta_{\Lambda^c}$ and $\tilde\tau:=\tau_\Lambda\xi_{\Lambda^c}$.
Then $\tilde\sigma\wedge\tilde\tau=(\sigma_\Lambda\wedge\tau_\Lambda)\eta_{\Lambda^c}$ and
$\tilde\sigma\vee\tilde\tau=(\sigma_\Lambda\vee\tau_\Lambda)\xi_{\Lambda^c}$, hence the lattice condition yields
\[
w_\Lambda^\xi(\sigma_\Lambda\vee\tau_\Lambda)\,w_\Lambda^\eta(\sigma_\Lambda\wedge\tau_\Lambda)
\ \ge\
w_\Lambda^\eta(\sigma_\Lambda)\,w_\Lambda^\xi(\tau_\Lambda),
\]
which is Holley's criterion. Therefore, by Holley's theorem \cite{G},
$\gamma_\Lambda(\cdot\mid \eta)\le_{\mathrm{st}}\gamma_\Lambda(\cdot\mid \xi)$, i.e.\ the specification is attractive.
\end{proof}

%
%

On a tree, Gibbs measures (and splitting Gibbs measures) can be expressed through
\emph{boundary laws}.
Let $m$ denote a strictly positive message on a child subtree, i.e.\ a positive
vector $m(\cdot)$ on the child's spin-set. The parent marginal is obtained by
a kernel transform
\begin{equation}\label{eq:kernel-transform}
(\mathcal{K}m)(s)\ :=\ \sum_{t} K(s,t)\,m(t),
\end{equation}
and products over independent subtrees.

A convenient order on positive vectors over totally ordered spin sets is the
\emph{monotone likelihood ratio} (MLR) order: $m\preceq_{\mathrm{MLR}} m'$
if $m'(t)/m(t)$ is increasing in $t$.

\begin{proposition}\label{prop:mlr}
Let $K$ be TP2 on $S\times T$. Then the transform $m\mapsto \mathcal{K}m$
preserves the MLR order:
\[
m\preceq_{\mathrm{MLR}} m'\quad\Longrightarrow\quad
\mathcal{K}m\preceq_{\mathrm{MLR}} \mathcal{K}m'.
\]
Consequently, on a tree with ferromagnetic kernel $K(s,t)=e^{\beta J st}$ ($J>0$),
the parent message (and hence any likelihood ratio at the parent) is monotone
in each child message/boundary condition.
\end{proposition}

\begin{proof}
Let $m,m'>0$ be vectors on $T$ and write $r(t):=m'(t)/m(t)$. The assumption
$m\preceq_{\mathrm{MLR}} m'$ means that $r(t)$ is increasing in $t\in T$.
We must show that the ratio
\[
R(s):=\frac{(Km')(s)}{(Km)(s)}=\frac{\sum_{t\in T}K(s,t)m'(t)}{\sum_{t\in T}K(s,t)m(t)}
\]
is increasing in $s\in S$, i.e. $Km\preceq_{\mathrm{MLR}} Km'$.

Fix $s_1<s_2$ in $S$. Since all terms are strictly positive, it is equivalent to prove
\[
(Km')(s_2)\,(Km)(s_1)\;-\;(Km')(s_1)\,(Km)(s_2)\ \ge\ 0.
\]
Expand and rearrange as a double sum:
\begin{align*}
& (Km')(s_2)\,(Km)(s_1)-(Km')(s_1)\,(Km)(s_2) \\
&\qquad=\sum_{t,u\in T}\Bigl[K(s_2,t)K(s_1,u)-K(s_1,t)K(s_2,u)\Bigr]\,m'(t)\,m(u).
\end{align*}
Define $\Delta(t,u):=K(s_2,t)K(s_1,u)-K(s_1,t)K(s_2,u)$.
Note that $\Delta(u,t)=-\Delta(t,u)$. Hence antisymmetrizing gives
\begin{align*}
& (Km')(s_2)\,(Km)(s_1)-(Km')(s_1)\,(Km)(s_2) \\
&\qquad=\frac12\sum_{t,u\in T}\Delta(t,u)\,\bigl(m'(t)m(u)-m'(u)m(t)\bigr)
=\sum_{t>u}\Delta(t,u)\,\bigl(m'(t)m(u)-m'(u)m(t)\bigr).
\end{align*}
For $t>u$, the TP2 property of $K$ (applied with $(s_1,s_2)$ and $(u,t)$) yields
\[
K(s_1,u)\,K(s_2,t)\ \ge\ K(s_1,t)\,K(s_2,u),
\]
i.e. $\Delta(t,u)\ge 0$. On the other hand,
\[
m'(t)m(u)-m'(u)m(t)=m(t)m(u)\,\bigl(r(t)-r(u)\bigr)\ \ge\ 0
\qquad (t>u),
\]
because $r$ is increasing. Therefore every term in the sum over $t>u$ is
nonnegative, and the whole sum is $\ge 0$. This shows $R(s_2)\ge R(s_1)$, hence
$R$ is increasing and $Km\preceq_{\mathrm{MLR}} Km'$.

For the final claim on a tree, the parent message is (up to normalization) a product
of transforms from the children. If a single child message is replaced by a larger
one in MLR order, Proposition~6 shows the corresponding factor $Km$ increases in
MLR, and multiplication by the remaining fixed positive factors preserves the MLR
order (the likelihood ratio is multiplied by an $s$-independent factor). Thus the
parent message/likelihood ratio is monotone in each child message and hence in the
boundary condition in the ferromagnetic case $J>0$.
\end{proof}


\section{Translation-invariant splitting Gibbs measures (TISGMs)}\label{sec2}

In this section, we investigate the existence of translation-invariant splitting Gibbs measures (TISGMs) associated with the Ising model involving triple mixed spins $\frac{1}{2}$, $1$ and $\frac{3}{2}$  by analyzing the system of equations \eqref{67}.

In the sequel, we assume that \( X = X_x \), \( Y = Y_x \), \( Z = Z_x \), \( T = T_x \), \( U = U_x \) and \( N = N_x \) for all \( x \in \Gamma^k \).
A splitting Gibbs measure is said to be \emph{translation-invariant} if it corresponds to such constant solutions. Then the system of equations \eqref{67} transforms into the following system
\begin{equation}\label{68}
\left\{%
\begin{array}{llllll}
X=\left(\frac{Y+\theta+\theta^2\,Z}{\theta^2\,Y+\theta+Z}\right)^k, \\[4mm]
Y=\Big(\frac{\theta^6\,T+\theta^4+\theta^2\,U+N}{\theta^3(T+1+U+N)}\Big)^k, \\[4mm]
Z=\Big(\frac{T+\theta^2+\theta^4\,U+\theta^6\,N}{\theta^3(T+1+U+N)}\Big)^k, \\[4mm]
T=\Big(\frac{X+\theta^3}{\theta^2+\theta\,X}\Big)^k, \\[4mm]
U=\Big(\frac{\theta\,X+1}{X+\theta}\Big)^k, \\[4mm]
N=\Big(\frac{\theta^3\,X+1}{\theta^2+\theta\,X}\Big)^k.
\end{array}%
\right.
\end{equation}

For convenience, we introduce the redefinitions:
\[x=\sqrt[k]{X},\quad y=\sqrt[k]{Y},\quad z=\sqrt[k]{Z},\quad t=\sqrt[k]{T},\quad u=\sqrt[k]{U},\quad v=\sqrt[k]{N}.\]
Thus, the functional equation \eqref{68} transforms into:
\begin{equation}\label{69}
\left\{%
\begin{array}{llllll}
x=\frac{y^k+\theta+\theta^2\,z^k}{\theta^2\,y^k+\theta+z^k}, \\[4mm]
y=\frac{\theta^6\,t^k+\theta^4+\theta^2\,u^k+v^k}{\theta^3(t^k+1+u^k+v^k)}, \\[4mm]
z=\frac{t^k+\theta^2+\theta^4\,u^k+\theta^6\,v^k}{\theta^3(t^k+1+u^k+v^k)}, \\[4mm]
t=\frac{x^k+\theta^3}{\theta^2+\theta\,x^k}, \\[4mm]
u=\frac{\theta\,x^k+1}{x^k+\theta}, \\[4mm]
v=\frac{\theta^3\,x^k+1}{\theta^2+\theta\,x^k}.
\end{array}%
\right.
\end{equation}
By substituting equations  into the first equation of \eqref{69}, the system reduces to a single fixed-point equation of the form:
\begin{equation}\label{eq13}
x = f(x,\theta,k),
\end{equation}
where
\begin{equation}\label{eq14}
f(x,\theta,k):= \frac{p(x,\theta,k)+\theta+\theta^2\,q(x,\theta,k)}{\theta^2\,p(x,\theta,k)+\theta+q(x,\theta,k)},
\end{equation}
\begin{equation}\label{eq141}
p(x,\theta,k):=\left(\frac{\theta^6\left(\frac{x^k+\theta^3}{\theta^2+\theta\,x^k}\right)^k+\theta^4+\theta^2\left(\frac{\theta\,x^k+1}{x^k+\theta}\right)^k+\left(\frac{\theta^3x^k+1}{\theta^2+\theta x^k}\right)^k}{\theta^3\left(\frac{x^k+\theta^3}{\theta^2+\theta\, x^k}\right)^k+\theta^3+\theta^3\left(\frac{\theta\,x^k+1}{x^k+\theta}\right)^k+\theta^3\left(\frac{\theta^3\,x^k+1}{\theta^2+\theta\, x^k}\right)^k}\right)^k,
\end{equation}
\begin{equation}\label{eq142}
q(x,\theta,k):=\left(\frac{\left(\frac{x^k+\theta^3}{\theta^2+\theta x^k}\right)^k+\theta^2+\theta^4\left(\frac{\theta\, x^k+1}{x^k+\theta}\right)^k+\theta^6\left(\frac{\theta^3\,x^k+1}{\theta^2+\theta\, x^k}\right)^k}{\theta^3\left(\frac{x^k+\theta^3}{\theta^2+\theta\,x^k}\right)^k+\theta^3+\theta^3\left(\frac{\theta\, x^k+1}{x^k+\theta}\right)^k+\theta^3\left(\frac{\theta^3\,x^k+1}{\theta^2+\theta\,x^k}\right)^k}\right)^k.
\end{equation}

Fix $k\ge1$ and $\theta>0$, and for the sake of simplicity, we write
\[
   F(x) := f(x,\theta,k), \qquad x>0.
\]

\begin{lem}\label{lem:basic-F}
For every $k\ge1$ and $\theta>1$ the following hold:
\begin{enumerate}
\item[(i)] The map $F:(0,\infty)\to(0,\infty)$ is real-analytic and strictly
      increasing.
\item[(ii)] The limits
      \[
        L_0(\theta,k) := \lim_{x\downarrow0} F(x), \qquad
        L_\infty(\theta,k) := \lim_{x\uparrow\infty} F(x)
      \]
      exist and satisfy $0<L_0(\theta,k)<\infty$ and
      $0<L_\infty(\theta,k)<\infty$.
\item[(iii)]  One has
      \[
        F(1)=1,
      \]
      so $x=1$ is a fixed point of $F$ for all $(\theta,k)$.
\end{enumerate}
In particular, the image of $F$ is contained in a compact interval
\[
   I(\theta,k) := \bigl[ \min\{L_0(\theta,k),L_\infty(\theta,k)\},
                    \max\{L_0(\theta,k),L_\infty(\theta,k)\} \bigr]
   \subset (0,\infty).
\]
\end{lem}

\begin{proof} (i).
By \emph{(12)}, each of $t(x),u(x),v(x)$ is a rational function of $x^k$
with strictly positive denominator on $(0,\infty)$; hence they are $C^\infty$.
The same holds for $y(x)$ and $z(x)$ by \emph{(12)}, and therefore for
$p(x)=y(x)^k$ and $q(x)=z(x)^k$ in \emph{(15)}–\emph{(16)}.

The denominator of $F$,
\[
  \theta^2 p(x,\theta,k) + \theta + q(x,\theta,k),
\]
is strictly positive for all $x>0$, so $F$ is real-analytic and strictly
positive on $(0,\infty)$.

For $J>0$, Lemma~\ref{lem:Ktp2} gives TP2 of the edge kernel. By
Theorem~\ref{thm:tp2-attractive} and Proposition~\ref{prop:mlr}, increasing the
boundary-law ratio at children increases (in the MLR sense) the parent message,
hence increases the induced likelihood ratio at the parent.
Under translation invariance, all child ratios are equal to $x$, and the parent
ratio is exactly the update $f(x,\theta,k)$; therefore $f(\cdot,\theta,k)$ is strictly increasing.

(ii). The existence and finiteness of $L_0(\theta,k)$ and $L_\infty(\theta,k)$
follow from the fact that, as $x\to0$ or $x\to\infty$, each of
$t(x),u(x),v(x)$ converges to a finite positive limit (given by the leading
terms in the rational expressions), and therefore so do $y(x),z(x)$, hence
$p(x),q(x)$ and finally $F(x)$.

(iii). The identity $F(1)=1$ is checked directly by substituting $x=1$
into \emph{(12)}–\emph{(16)}, and corresponds to the disordered
translation-invariant splitting Gibbs measure.
\end{proof}

The TISGM corresponding to the solution $x_0=1$ is denoted by $\mu_0$, and it is called the \textit{disordered phase} of the model.

\begin{proposition}\label{prop:global-iter}
Fix $k\ge1$ and $\theta>1$, and let $F(x)=f(x,\theta,k)$ as above. For any
initial value $x_0>0$ consider the recursion
\[
   x_{n+1} = F(x_n), \qquad n\ge0.
\]
Then:
\begin{enumerate}
\item[(i)] The orbit $(x_n)_{n\ge0}$ is monotone from the first step onward:
      either $(x_n)$ is nondecreasing, or it is nonincreasing.
\item[(ii)] The sequence $(x_n)$ converges to a limit $\ell(x_0)\in(0,\infty)$, and it is a fixed point of $F$.
\end{enumerate}
\end{proposition}

\begin{proof} (i).
By Lemma~\ref{lem:basic-F}, $F$ is strictly increasing and maps $(0,\infty)$
into the bounded interval $I(\theta,k)$.

If $x_1\ge x_0$, then by monotonicity,
\[
  x_2 = F(x_1) \ge F(x_0)=x_1,\quad
  x_3 = F(x_2) \ge F(x_1)=x_2,\ \ldots,
\]
and hence $(x_n)$ is nondecreasing. If $x_1\le x_0$, the same argument with
reversed inequalities shows that $(x_n)$ is nonincreasing. In both cases,
$(x_n)$ is bounded (by $I(\theta,k)$), hence convergent to some
$\ell(x_0)\in I(\theta,k)$.

(ii) By continuity of $F$,
\[
  F(\ell(x_0)) = F\Bigl(\lim_{n\to\infty} x_n\Bigr)
               = \lim_{n\to\infty} F(x_n)
               = \lim_{n\to\infty} x_{n+1}
               = \ell(x_0),
\]
so $\ell(x_0)$ is a fixed point of $F$.
\end{proof}

\begin{proposition}\label{prop:linear-stability}
Let $x_*>0$ be a fixed point of $F$, i.e.\ $F(x_*)=x_*$. Then:
\begin{enumerate}
\item If $0\le F'(x_*)<1$, there exists $\delta>0$ such that for all
      $x_0\in(x_*-\delta,x_*+\delta)$ we have $x_n\to x_*$ as $n\to\infty$.
\item If $F'(x_*)>1$, there exists $\delta>0$ such that for all
      $x_0\in(x_*-\delta,x_*+\delta)\setminus\{x_*\}$, the orbit $(x_n)$
      eventually leaves $(x_*-\delta,x_*+\delta)$ and does not converge to $x_*$.
\end{enumerate}
\end{proposition}

The proof is evident.

\begin{cor}\label{cor:stability-1}
The point $x=1$ is a fixed point of $F$ for all $k\ge1$ and $\theta>1$.
Let
\[
  s_k(\theta) := F'(1) - 1 = f'(1,\theta,k)-1.
\]
Then:
\begin{enumerate}
\item If $s_k(\theta)<0$, the fixed point $x=1$ is locally attracting.
\item If $s_k(\theta)>0$, the fixed point $x=1$ is locally repelling.
\item If $s_k(\theta)=0$, the fixed point is linearly neutral (critical).
\end{enumerate}
\end{cor}

We arrive our first main result.

\begin{theorem}\label{thm10}
Let $k\ge2$ and $J>0$, and put $\theta=\exp(\beta J/2)$.
Consider the triple mixed-spin Ising model with spin values
$\bigl(\tfrac12,1,\tfrac32\bigr)$ on the Cayley tree $\Gamma_k$, where the
spin species $\Psi,\Phi,\Upsilon$ are arranged periodically with period three
along the generations.

\begin{enumerate}
\item For all $(\theta,k)$ there exists at least one fixed point
      $x_*=1$, corresponding to the disordered TISGM~$\mu_0$.
 \item
      If $s_k(\theta)>0$, then the fixed-point equation
\[
  x \;=\; f(x,\theta,k)
\]
admits at least three distinct positive solutions. Consequently, for every such
pair $(\theta,k)$ the model admits at least three $\mu_0$,$\mu_1$ and $\mu_2$ distinct translation-invariant
splitting Gibbs measures.
\end{enumerate}
\end{theorem}


Our sufficient condition \(s_k(\theta)>0\) for the existence of at least three TISGMs is in the same spirit as linear-instability criteria used in earlier mixed--spin analyses on trees (see e.g. \cite{akin2018gibbs,akin20231,akin20232,akin2024}), where bifurcation from the symmetric fixed point is detected via the derivative of the corresponding recursion. The novelty of Theorem~\ref{thm10} lies in handling the period--3 triple--spin arrangement \((\tfrac12,1,\tfrac32)\): the scalar map \(f\) in \eqref{eq14} encodes nested interactions across three layers and produces a richer dependence on \(\theta\) and the tree order \(k\) than in the two-layer mixed--spin models of \cite{akin2018gibbs}. Numerically (Figures~\ref{fig111},\ref{fig112},\ref{fig113}) we observe that the instability region \(\{\theta:s_k(\theta)>0\}\) enlarges with \(k\), indicating that the period--3 inhomogeneity facilitates multiplicity over a wider parameter range than the simpler mixed--spin settings treated previously.

\begin{figure}[h]
\centering
\includegraphics[width=12cm]{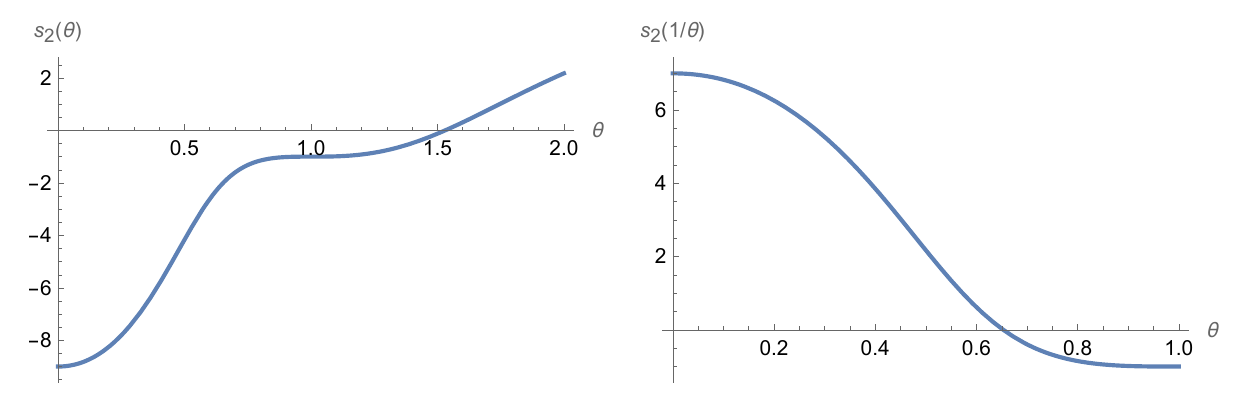}
\caption{Plots of $s_2(\theta)$ for $\theta \in (0,2)$ (left) and $s_2(1/\theta)$ for $\theta \in (0,1)$ (right). The change of variables $\theta \mapsto 1/\theta$ allows us to visualize
and confirm the monotonicity of $s_2(\theta)$ on $(1,\infty)$.}\label{fig111}
\end{figure}
\begin{figure}[t]
\centering
\includegraphics[width=0.7\textwidth]{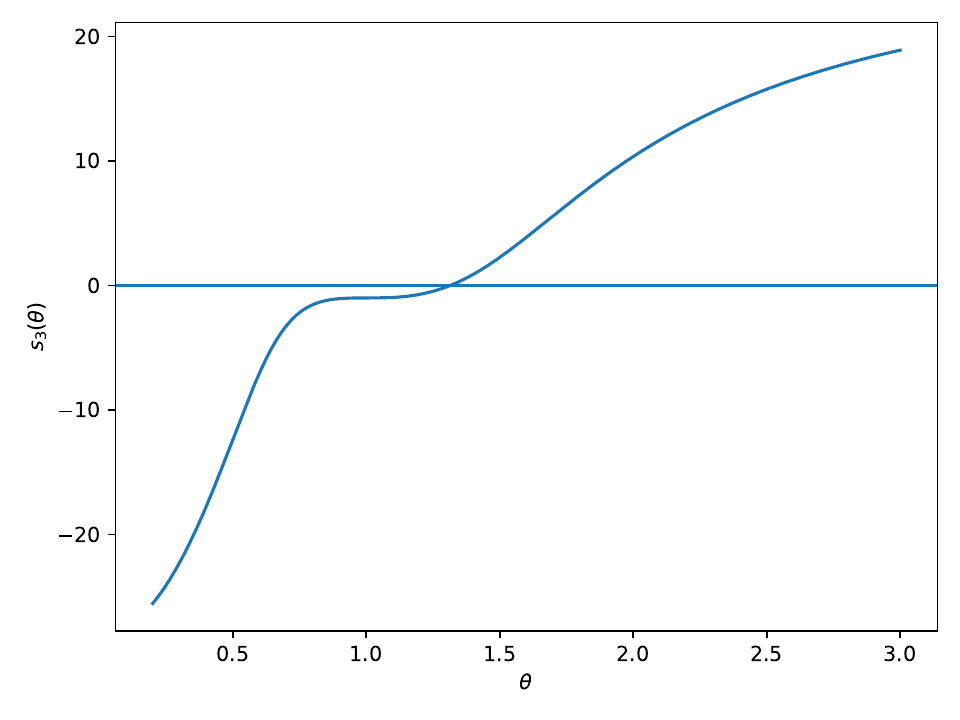}
\caption{Plot of $s_3(\theta)=f'(1,\theta,3)-1$. The region where $s_3(\theta)>0$
corresponds to the existence of at least three translation-invariant splitting
Gibbs measures.}\label{fig112}
\end{figure}
\begin{figure}[h]
\centering
\includegraphics[width=12cm]{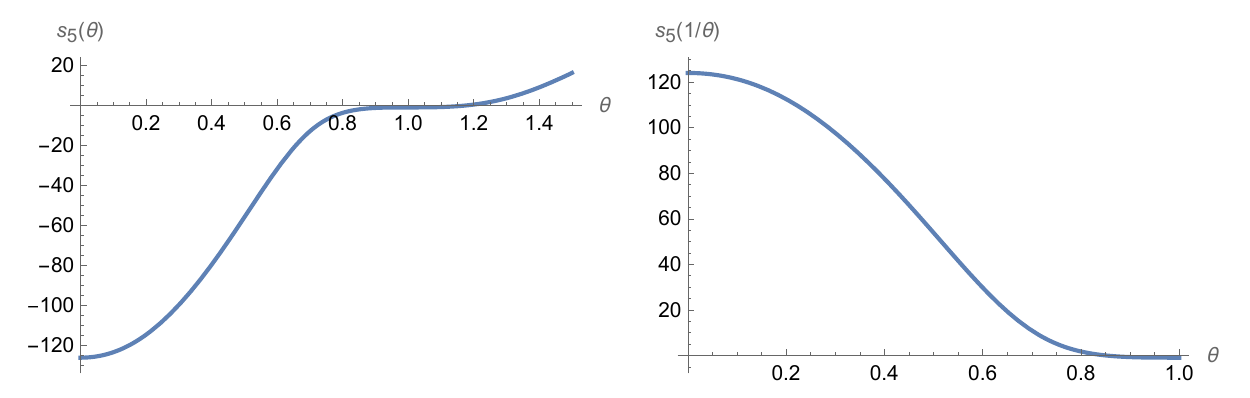}
\caption{Graphs of $s_5(\theta)$ for $\theta \in (0,1.5)$ (left) and
$s_5(1/\theta)$ for $\theta \in (0,1)$ (right).}\label{fig113}
\end{figure}

For instance, in the case $k=3$, numerical evaluation of $s_3(\theta)$ shows that there exists a critical value $\theta_c \approx 1.3$ such that $s_3(\theta)>0$ for all $\theta>\theta_c$, confirming the onset of multiple TISGMs. Numerical evaluation indicates that for $k=3$ the function $s_3(\theta)$ becomes positive above a critical value $\theta_c$, which signals the instability of the
disordered fixed point and the appearance of multiple TISGMs.

Here is a table of critical values of $\theta_c$.\\
\begin{tabular}{rr}
\toprule
$k$ & $\theta_c$ \\
\midrule
2 & 1.47626086 \\
3 & 1.29940412 \\
4 & 1.23599394 \\
5 & 1.17737002 \\
\bottomrule
\end{tabular}

\section{Extremality of ordered TISGMs for $k=2$}

In this section we restrict attention to the Cayley tree of order $k=2$
and study the extremality of the ordered translation-invariant splitting
Gibbs measures (TISGMs) arising from the scalar recursion
\begin{equation}\label{eq:scalar-k2}
  x = f(x,\theta,2),
\end{equation}
where $f$ is given by \eqref{eq14}.

In what follows, as before, we assume
 that $J>0$ and hence $\theta = \exp(\beta J/2) > 1$.

For notational convenience we write
\[
   F_2(x) := f(x,\theta,2), \qquad x>0.
\]
By Lemma~\ref{lem:basic-F} (i) the map $x\mapsto F_2(x,\theta)$ is strictly increasing on
$(0,\infty)$ for all $\theta>1$, and satisfies
\[
   F_2(0,\theta) > 0, \qquad \lim_{x\to\infty} F_2(x,\theta) < \infty,
   \qquad F_2(1,\theta)=1.
\]

Throughout the section, we always suppose $s_2(\theta)>0$, which indicates the function
$F_2(x)$ admits three distinct solutions $x_-(\theta), x_0=1, x_+(\theta)$ such that
$x_-(\theta)<1<x_+(\theta)$.

We interpret $x_-(\theta)$ and $x_+(\theta)$ as the \emph{minimal} and
\emph{maximal} boundary-law ratios at the $\Psi$-levels (spins
$\pm\frac12$) compatible with translation invariance, cf. the reduction
of \eqref{67} to the scalar equation
\eqref{eq:scalar-k2}.

Recall that for $J>0$ the two-site kernel
$K(s,t)=\exp(\beta J s t)$ is totally positive of order~2 (TP2), hence
the Gibbs specification is attractive in the FKG sense.
In particular, the finite-volume Gibbs measures with deterministic
boundary condition $\xi_{\partial V_n}$ are monotone in $\xi_{\partial
  V_n}$ with respect to the natural partial order on spin configurations.

Let $\xi^+_{\partial V_n}$ (resp.\ $\xi^-_{\partial V_n}$) denote the
boundary condition where all boundary spins are fixed to their maximal
(resp.\ minimal) values:
\[
  \xi^+_{\partial V_n}(x) =
   \begin{cases}
     +\tfrac12, & x\in\Gamma^k_0\cap W_n,\\
     +1,        & x\in\Gamma^k_1\cap W_n,\\
     +\tfrac32, & x\in\Gamma^k_2\cap W_n,
   \end{cases}
   \qquad
  \xi^-_{\partial V_n}(x) =
   \begin{cases}
     -\tfrac12, & x\in\Gamma^k_0\cap W_n,\\
     -1,        & x\in\Gamma^k_1\cap W_n,\\
     -\tfrac32, & x\in\Gamma^k_2\cap W_n,
   \end{cases}
\]
where $k=2$ in our setting.

Denote by $\mu^{+,n}$ and $\mu^{-,n}$ the corresponding finite-volume
Gibbs measures on $V_n$, and let $\mu^+$ and $\mu^-$ be their weak
limits along an increasing sequence of volumes:
\[
  \mu^\pm := \lim_{n\to\infty} \mu^{\pm,n},
\]
where the limit exists by monotonicity and compactness of the space of
probability measures. Both limits are Gibbs measures for the infinite
volume system.

\begin{proposition}\label{prop:min-max-G}
Assume $J>0$ and  consider the Ising model with triple mixed spins
$\bigl(\tfrac12,1,\tfrac32\bigr)$ on a Cayley tree $\Gamma_2$ of order $k=2$. Let $\preccurlyeq$ denote the
coordinatewise stochastic order on probability measures induced by the natural
partial order on the spin space. Then:
\begin{enumerate}
\item The weak limits
      \[
         \mu^- := \lim_{n\to\infty}\mu^{-,n}, \qquad
         \mu^+ := \lim_{n\to\infty}\mu^{+,n},
      \]
      where $\mu^{\pm,n}$ are the finite-volume Gibbs measures in $V_n$ with
      all boundary spins fixed to their minimal (resp.\ maximal) values, exist
      and are Gibbs measures for the model.
\item The measures $\mu^-$ and $\mu^+$ are splitting Gibbs measures and are
      translation-invariant with the same period-$3$ structure as the
      interaction (i.e.\ invariant under all automorphisms of $\Gamma_2$ which
      preserve the decomposition $\Gamma^2_0,\Gamma^2_1,\Gamma^2_2$).
\item For any Gibbs measure $\mu$ of the model one has
      \[
         \mu^- \;\preccurlyeq\; \mu \;\preccurlyeq\; \mu^+.
      \]
      Hence $\mu^-$ and $\mu^+$ are the minimal and maximal Gibbs measures in
      this stochastic order.
\end{enumerate}
\end{proposition}

\begin{remark}
Proposition~\ref{prop:min-max-G} do not rely on the number of fixed points of
the scalar map $f_2(\cdot,\theta)$. They hold for all $\theta>1$ in the
ferromagnetic regime, independently of whether the equation $x=f_2(x,\theta)$
has one or several solutions. The assumption that $f_2(\cdot,\theta)$ has
exactly three positive fixed points is only used later, in
Theorem~\ref{thm:dichotomy-k2}, to conclude that the set of
translation-invariant splitting Gibbs measures is
$\{\mu_-,\mu_0,\mu_+\}$ and that these three measures are distinct.
\end{remark}

Using the boundary-law representation from Theorem~\ref{thm1},
each TISGM corresponds to a constant solution of system~\eqref{68},
equivalently to a fixed point of the scalar map $f_2(\cdot,\theta)$.

\begin{lem}\label{lem:pm-fixed-points}
Let $J>0$, $k=2$ and $\theta=e^{\beta J/2}>1$.
Let $\mu^-$ and $\mu^+$ be the minimal and maximal Gibbs measures from
Proposition~\ref{prop:min-max-G}. Then:
\begin{enumerate}
\item The measures $\mu^-$ and $\mu^+$ are translation-invariant splitting
      Gibbs measures corresponding to constant boundary laws whose
      $\Psi$-component likelihood ratios are $x_-(\theta)$ and $x_+(\theta)$,
      respectively. Equivalently,
      \[
        \mu^- = \mu_{x_-(\theta)}, \qquad
        \mu^+ = \mu_{x_+(\theta)}.
      \]
  Moreover, the measures $\mu^-$ and $\mu^+$ are extremal Gibbs measures.
\end{enumerate}
\end{lem}

\begin{theorem}
\label{thm:dichotomy-k2}
Consider the triple mixed Ising model with spins
$(\tfrac12,1,\tfrac32)$ on a Cayley tree of order $k=2$ and coupling
$J>0$, and let $\theta=e^{\beta J/2}>1$.
Let $\mathcal G_T^{(2)}$ denote the set of translation-invariant
splitting Gibbs measures for this model.

\begin{enumerate}
\item Suppose that the fixed-point equation \eqref{eq:scalar-k2} has the
      unique solution $x=1$ (in particular, this holds whenever
      $s_2(\theta)\le 0$ and no additional fixed points appear). Then
      $\mathcal G_T^{(2)} = \{\mu_0\}$, where $\mu_0$ is the disordered
      TISGM corresponding to $x=1$.
\item Suppose that the fixed-point equation \eqref{eq:scalar-k2} has
      exactly three positive solutions
      \[
          x_-(\theta) < 1 < x_+(\theta).
      \]
      Then
      \[
         \mathcal G_T^{(2)} = \{\mu_-,\mu_0,\mu_+\},
      \]
      where $\mu_-$ and $\mu_+$ are the TISGMs associated with
      $x_-(\theta)$ and $x_+(\theta)$ via
      Lemma~\ref{lem:pm-fixed-points}, and $\mu_-$ and $\mu_+$ are
      extremal Gibbs measures in the full convex set of Gibbs measures.
\end{enumerate}
\end{theorem}

\begin{proof}
If $x=1$ is the only positive fixed point, then the corresponding
boundary law is unique, hence the only TISGM is $\mu_0$.

In the second case, Theorem~\ref{thm10} shows that
$x_-(\theta)$ and $x_+(\theta)$ are the minimal and maximal fixed
points. Lemma~\ref{lem:pm-fixed-points} identifies these with the
plus/minus TISGMs $\mu^\pm$, which we denote by $\mu_-$ and $\mu_+$
for brevity. Any other TISGM must correspond to $x=1$, giving $\mu_0$.

Extremality of $\mu_-$ and $\mu_+$ as Gibbs measures follows from the
general theory of attractive spin systems: for an attractive
nearest-neighbour specification on a tree, the minimal and maximal Gibbs
measures are extremal (see, \cite{G,Ro}). In our setting,
$\mu_-$ and $\mu_+$ are exactly these minimal and maximal measures by
Proposition~\ref{prop:min-max-G}, hence they belong to $\mathrm{ex}(G)$,
the set of extremal Gibbs measures.
\end{proof}

\begin{remark}
For $k=2$, numerical evaluation of $s_2(\theta)$ indicates that there is
a unique critical value $\theta_c\approx 1.47626$ such that
$s_2(\theta)\le 0$ for $\theta\le\theta_c$ and $s_2(\theta)>0$ for
$\theta>\theta_c$, see Figure~2.
Combined with Theorem~\ref{thm:dichotomy-k2}, this suggests the following
picture for the binary tree:
\[
  \begin{cases}
    \text{unique TISGM } \mu_0, & \theta\le\theta_c,\\[0.3em]
    \text{three TISGMs } \mu_-,\mu_0,\mu_+, & \theta>\theta_c,
  \end{cases}
\]
with $\mu_\pm$ extremal ordered phases. A rigorous proof that
\eqref{eq:scalar-k2} has exactly three positive solutions when
$s_2(\theta)>0$ would require a more detailed analysis of the rational
map $f_2(\cdot,\theta)$; we leave this as an interesting open problem.
\end{remark}

\section{Non-extremality of $\mu_0$}\label{sec4}

In this section, we derive a sufficient condition for the non-extremality of the disordered phase $\mu_0$ based on the Kesten-Stigum criterion.

Recall that a Gibbs measure is said to be extremal if it cannot be represented as a non-trivial convex combination of other Gibbs measures; otherwise, it is called non-extremal \cite{G}. A tree-indexed Markov chain \(X: V \to \Phi\) is defined by selecting \(X(x^{(0)})\) at the root according to \(\nu\), and then assigning values to each vertex \(v \neq x^{(0)}\) using the transition probabilities, given the value at its parent, independently of everything else (see \cite{G} for details).

We define the transition probability matrix $\mathcal{P} = (P_{ij})$ by
\[
P_{ij} = \frac{\exp\big(\beta J\, i\, j + \ddot{h}_j\big)}
{\sum\limits_{u \in \{-1,0,1\}} \exp\big(\beta J\, i\, u + \ddot{h}_u\big)},
\]
where $i \in \{-\tfrac{1}{2}, \tfrac{1}{2}\}$ and $j \in \{-1,0,1\}$. Introducing the
parameters
\[
Y = \exp(\ddot{h}_{-1} - \ddot{h}_0), \qquad
Z = \exp(\ddot{h}_{1} - \ddot{h}_0),
\]
the matrix $\mathcal{P}$ can be written in the form
\begin{equation}\label{eq28a}
\mathcal{P} =
\begin{pmatrix}
\dfrac{\theta^2 Y}{\theta^2 Y + \theta + Z} &
\dfrac{\theta}{\theta^2 Y + \theta + Z} &
\dfrac{Z}{\theta^2 Y + \theta + Z} \\[2ex]
\dfrac{Y}{Y + \theta + \theta^2 Z} &
\dfrac{\theta}{Y + \theta + \theta^2 Z} &
\dfrac{\theta^2 Z}{Y + \theta + \theta^2 Z}
\end{pmatrix}.
\end{equation}

Next, we define the transition matrix $\mathcal{Q} = (Q_{ij})$ by
\[
Q_{ij} = \frac{\exp\big(\beta J\, i\, j + \widehat{h}_j\big)}
{\sum\limits_{u \in \{\pm \frac{3}{2}, \pm \frac{1}{2}\}}
\exp\big(\beta J\, i\, u + \widehat{h}_u\big)},
\]
where $i \in \{-1,0,1\}$ and
$j \in \{-\tfrac{3}{2}, -\tfrac{1}{2}, \tfrac{1}{2}, \tfrac{3}{2}\}$.
Introducing the variables
\[
T = \exp(\widehat{h}_{-\frac{3}{2}} - \widehat{h}_{-\frac{1}{2}}), \quad
U = \exp(\widehat{h}_{\frac{1}{2}} - \widehat{h}_{-\frac{1}{2}}), \quad
V = \exp(\widehat{h}_{\frac{3}{2}} - \widehat{h}_{-\frac{1}{2}}),
\]
the matrix $\mathcal{Q}$ takes the form
\begin{equation}\label{eq29a} \mathcal{Q} = \begin{pmatrix} \frac{\theta^6T}{\theta^6T + \theta^4 + \theta^2U+V} & \frac{\theta^2}{\theta^6T + \theta^4 + \theta^2U+V} & \frac{\theta^2U}{\theta^6T + \theta^4 + \theta^2U+V} & \frac{V}{\theta^6T + \theta^4 + \theta^2U+V}\\ \frac{T}{T + 1 + U+V} & \frac{1}{T + 1 + U+V} & \frac{U}{T + 1 + U+V} & \frac{V}{T + 1 + U+V} \\ \frac{T}{T+\theta^2+\theta^4U+\theta^6V} & \frac{\theta^2}{T+\theta^2+\theta^4U+\theta^6V} & \frac{\theta^4U}{T+\theta^2+\theta^4U+\theta^6V} & \frac{\theta^6V}{T+\theta^2+\theta^4U+\theta^6V} \end{pmatrix}. \end{equation}

Finally, we introduce the transition matrix $\mathcal{R} = (R_{ij})$ defined by
\[
R_{ij} = \frac{\exp\big(\beta J\, i\, j + \tilde{h}_j\big)}
{\sum\limits_{u \in \{\pm \frac{1}{2}\}}
\exp\big(\beta J\, i\, u + \tilde{h}_u\big)},
\]
where $i \in \{-\tfrac{3}{2}, -\tfrac{1}{2}, \tfrac{1}{2}, \tfrac{3}{2}\}$ and
$j \in \{-\tfrac{1}{2}, \tfrac{1}{2}\}$.
Setting
\[
X = \exp(\tilde{h}_{\frac{1}{2}} - \tilde{h}_{-\frac{1}{2}}),
\]
the matrix $\mathcal{R}$ assumes the form
\begin{equation}\label{eq30a}
\mathcal{R} =
\begin{pmatrix}
\dfrac{\theta^3}{\theta^3 + X} & \dfrac{X}{\theta^3 + X} \\[1ex]
\dfrac{\theta}{\theta + X} & \dfrac{X}{\theta + X} \\[1ex]
\dfrac{1}{1 + \theta X} & \dfrac{\theta X}{1 + \theta X} \\[1ex]
\dfrac{1}{1 + \theta^3 X} & \dfrac{\theta^3 X}{1 + \theta^3 X}
\end{pmatrix}.
\end{equation}
Hence, any vector $v = (X, Y, Z, T, U, V) \in \mathbb{R}^6$ satisfying system~\eqref{68} defines a TISGM, which in turn induces a tree-indexed non-homogeneous Markov chain.

Note that the matrices \(\mathcal{P}\), \(\mathcal{Q}\), and \(\mathcal{R}\) are stochastic, meaning that each row sums to \(1\). By multiplying two stochastic matrices, one obtains another stochastic matrix. Consequently, the product of the three matrices \(\mathcal{P}\), \(\mathcal{Q}\), and \(\mathcal{R}\) also yields a stochastic matrix:
\begin{equation}\label{H}
\mathcal{H} = \mathcal{P}\mathcal{Q}\mathcal{R}.
\end{equation}



\paragraph{Kesten--Stigum criterion for the $3$--step channel.}
Recall that the level-to-level transitions are described by the stochastic matrices
$ \mathcal{P}:\Psi\to\Phi$, $ \mathcal{Q}:\Phi\to\Upsilon$, and $ \mathcal{R}:\Upsilon\to\Psi$, and that \eqref{H}
is the induced $\Psi\to\Psi$ transition matrix along a path of length $3$.
If we restrict attention to the subsequence of levels $\Psi_{0},\Psi_{3},\Psi_{6},\ldots$,
then each $\Psi$--vertex has exactly $k^3$ descendants at the next $\Psi$--level. Equivalently,
the broadcast process on the compressed tree (keeping only $\Psi$--levels) has branching factor $k^3$
and channel $H$.
Therefore, the Kesten--Stigum bound (see, e.g., \cite{KS}) yields a sufficient condition for
non-extremality (reconstruction) of the corresponding splitting Gibbs measure in terms of the
second eigenvalue $\lambda(H)$:
\begin{equation}\label{KS-3step}
k^{3}\,\lambda(H)^{2}\ >\ 1.
\end{equation}

Due to the complexity of the solutions of the system \eqref{68}, an explicit expression for this eigenvalue is rather cumbersome and lengthy.

We observe that the system \eqref{68} admits the solution
\[
X_0 = U_0 = 1,~Y_0 = Z_0 = \left(\frac{\theta^6 T_0 + \theta^4 + \theta^2 + T_0}{2\theta^3(T_0 + 1)}\right)^k,~V_0 = T_0 = \left(\frac{1 + \theta^3}{\theta^2 + \theta}\right)^k.
\]
The corresponding TISGM $\mu_0$ is referred
to as the \emph{disordered phase} of the model. Let $\lambda_{0,k}$ denote the second-largest eigenvalue (in absolute value) of the transition matrix \(\mathcal{H}\) associated with this solution.


\noindent
Let $\lambda_{0,k}=\lambda_{0,k}(\theta)$ denote the (non-trivial) eigenvalue of \eqref{H} in the
disordered solution. We set
\begin{equation}\label{def-gk}
g_k(\theta)\ :=\ k^{3}\,\lambda_{0,k}(\theta)^{2}-1.
\end{equation}
In particular, if $g_k(\theta)>0$, then \eqref{KS-3step} holds and the disordered splitting Gibbs
measure $\mu_0$ is non-extremal.



\begin{theorem}\label{thm15}
Let $k\ge 1$ and $\theta>1$. Consider the disordered translation-invariant splitting Gibbs measure
$\mu_0$ on the Cayley tree of order $k$, and let $\mathcal{H}$ be given by \eqref{H} which corresponds to $\mu_0$. If
\begin{equation}\label{thm15-KS}
k^{3}\,\lambda_{0,k}(\theta)^{2}\ >\ 1,
\end{equation}
equivalently $g_k(\theta)>0$ with $g_k$ defined by \eqref{def-gk}, then $\mu_0$ is {non-extremal}.
\end{theorem}



\begin{remark}\label{rk1}
Observe that the set of parameters for which the inequality $g_k(\theta)>0$ holds is nonempty, as illustrated in Fig.~\ref{fig115}.
\end{remark}
\begin{figure}[h]
\centering
\includegraphics[width=11cm]{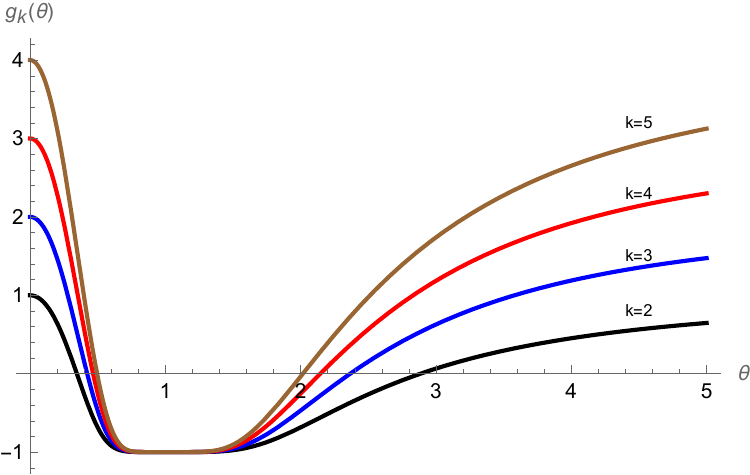}
\caption{The plots of $g_k(\theta)$ for $\theta \in (0,5)$ and $k=2,3,4,5$ are shown. It is observed that the interval of $\theta$ for which $g_k(\theta)>0$ increases as $k$ becomes larger.}\label{fig115}
\end{figure}

Combining Lemma~\ref{lem:pm-fixed-points} with the analysis of the disordered
phase in Section~5, we obtain the following picture for $k=2$. Recall that the
tree-indexed Markov chain associated with the disordered TISGM $\mu_0$ has
effective transition matrix $H(\theta,2)$ on the alphabet $\Psi$, with second
eigenvalue $\lambda_{0,2}$, and that the Kesten--Stigum quantity
\[
   g_2(\theta) := 8\,\lambda_{0,2}^2 - 1
\]
governs non-extremality of $\mu_0$ (see Theorem~\ref{thm15}). Whenever
$g_2(\theta)>0$ the disordered TISGM $\mu_0$ is non-extremal, whereas the ordered
TISGMs $\mu_-$ and $\mu_+$ remain extremal by
Theorem~\ref{thm:dichotomy-k2} and Proposition~\ref{prop:min-max-G}. In particular,
in the parameter region where both a phase transition (multiple fixed points of
$f_2(\cdot,\theta)$) and the KS condition $g_2(\theta)>0$ hold, the only extremal
translation-invariant splitting Gibbs measures are the ordered phases $\mu_-$ and
$\mu_+$, while the disordered phase $\mu_0$ admits a non-trivial ergodic
decomposition.

\begin{remark}\label{rk01} We finally comment on further extremality conditions for TISGMs. Several methods in the literature provide sufficient conditions for extremality based on finite-dimensional optimization of the transition matrix, including the percolation approach \cite{Martinelli,Mossel,M22}, the symmetric-entropy method \cite{FormentinKulske2009}, and bounds for Ising models with external fields \cite{Borgs2006}. However, in our setting the transition matrix depends on solutions with highly complex structure, which makes these methods difficult to apply, especially when explicit solutions are unavailable. Nevertheless, our results provide a basis for future numerical studies of extremality.
\end{remark}

\section{Conclusion}

In this work we investigated an Ising model with triple mixed spins
$(1/2,1,3/2)$ on a Cayley tree of arbitrary order $k\ge1$, where the spin
values are distributed periodically with period three along the generations.
Within the rigorous framework of splitting Gibbs measures, we derived the
exact boundary-law compatibility system and analyzed its translation-invariant
solutions. Recent studies have clarified that the recursive equation systems
obtained from the cavity method and from the Kolmogorov consistency
condition are in fact identical; in the present context, the constant solutions
of the boundary-law system are therefore in one-to-one correspondence with
translation-invariant Gibbs measures of the model \cite{akin20232,A24CSF}.

A first main outcome of our analysis is an explicit sufficient condition for
phase coexistence. In the ferromagnetic regime $J>0$, we reduced the problem
of translation-invariant splitting Gibbs measures to a scalar fixed-point
equation $x=f(x,\theta,k)$ for a rational map $f$ depending on the tree order
$k$ and the parameter $\theta=\exp(\beta J/2)$. The derivative
$s_k(\theta)=f'(1,\theta,k)-1$ at the disordered fixed point $x=1$ controls
the local stability of the corresponding TISGM $\mu_0$. We proved that
$s_k(\theta)>0$ implies the existence of at least three positive fixed points
of $f$, and hence at least three distinct TISGMs. This reveals a phase
transition driven by the periodic inhomogeneity of the spin structure.
Numerical computations indicate that the instability region
$\{\theta:s_k(\theta)>0\}$ expands with the tree order $k$, showing that
higher branching enhances the tendency towards multiplicity of TISGMs.

For the binary tree $k=2$ we went further and made systematic use of
attractiveness and TP2 of the ferromagnetic kernel. We constructed plus and
minus Gibbs measures as weak limits with maximal and minimal boundary
conditions, proved that these limits $\mu^+$ and $\mu^-$ are
translation-invariant splitting Gibbs measures with the same period-$3$
structure as the interaction, and showed that they are the minimal and
maximal Gibbs measures in the natural stochastic order. Using the boundary-law
representation, we identified $\mu^\pm$ with the extremal fixed points
$x^\pm(\theta)$ of the scalar recursion and established their extremality in
the full convex set of Gibbs measures. When the scalar equation admits
exactly three positive fixed points $x^-(\theta)<1<x^+(\theta)$, we obtained
a complete description of translation-invariant splitting Gibbs measures on
the binary tree:
\[
\mathcal G^{(2)}_{\mathrm{T}} = \{\mu^-,\mu_0,\mu^+\},
\]
with $\mu^\pm$ extremal ordered phases and $\mu_0$ the disordered phase.

A second main contribution concerns the extremality of the disordered phase.
For any TISGM we constructed the associated tree-indexed Markov chain on
$\Psi$-spins and derived the effective $2\times2$ transition matrix
$H(\theta,k)$ as a product of three level-wise stochastic matrices.
Applying the Kesten--Stigum criterion \cite{KS} to $H(\theta,k)$, we obtained a
non-extremality condition of the form $g_k(\theta)=k^3\lambda_{0,k}^2-1>0$,
where $\lambda_{0,k}$ is the second eigenvalue of $H(\theta,k)$.
This yields a nonempty parameter region in which the disordered TISGM
$\mu_0$ is non-extremal. For $k=2$, combining the Kesten--Stigum condition
with the classification of TISGMs shows that in the regime where both
$s_2(\theta)>0$ and $g_2(\theta)>0$ hold, the only extremal
translation-invariant splitting Gibbs measures are the ordered phases
$\mu^-$ and $\mu^+$, whereas the disordered phase $\mu_0$ decomposes into
a non-trivial mixture of extremal states.

Beyond existence and extremality, we also derived explicit thermodynamic
quantities associated with a TISGM. In particular, we expressed the finite­
volume free energy and the specific free energy, the sublattice magnetizations
on the three level classes of the tree, and the magnetization and susceptibility
along a typical ray in terms of the underlying translation-invariant boundary
law. These formulas make it possible to connect the probabilistic description
of Gibbs measures to observable thermodynamic behavior and to quantify the
impact of the period-$3$ inhomogeneity on macroscopic quantities.

Several open problems remain. For the binary tree, a complete rigorous
analysis of the global fixed-point structure of the rational map
$x\mapsto f(x,\theta,2)$, beyond local stability \cite{fmok}, would allow one to turn the
numerically observed picture of exactly three positive fixed points for
$s_2(\theta)>0$ into a theorem. For higher orders $k\ge3$, the classification
of TISGMs and the extremality of ordered phases remain largely open.
It would also be interesting to complement the Kesten--Stigum approach with
other extremality criteria, such as percolation-type bounds, symmetric-entropy
methods, or techniques tailored to models with external fields, in order to
sharpen the reconstruction/non-reconstruction threshold for the present
triple mixed-spin system. Finally, extensions to more general periodic spin
patterns, to models with competing interactions, or to irregular tree
structures (such as $(m,k)$-ary trees) represent promising directions for
future research (see \cite{QMA25}).

Overall, the analysis developed here contributes to the broader program of
understanding mixed-spin Ising models on non-amenable graphs. It highlights
how deterministic inhomogeneity at the microscopic level can be encoded in
boundary-law recursions and can produce rich phase diagrams, intricate
extremality properties, and nontrivial thermodynamic behavior on tree-like
structures \cite{ART,R15}.

\section*{Acknowledgments}
	The first author thanks the UAEU UPAR Grant No. G00004962 for support.

\section*{Declaration} The authors declare that they have no conflict of~interests.

\section*{Availability of data and material} Not applicable.

\appendix

\section{Proof of Proposition \ref{prop:min-max-G}}
Recall the spin sets at the three types of levels
\[
   \Psi = \Bigl\{-\tfrac12,\tfrac12\Bigr\}, \quad
   \Phi = \{-1,0,1\}, \quad
   \Upsilon = \Bigl\{-\tfrac32,-\tfrac12,\tfrac12,\tfrac32\Bigr\}
\]
assigned to $\Gamma^2_0,\Gamma^2_1,\Gamma^2_2$, respectively. The full
configuration space is
\[
   \Xi = \Omega_0\times\Omega_1\times\Omega_2,
\]
with $\Omega_i=S_i^{\Gamma^2_i}$ as in Section~2. :contentReference[oaicite:1]{index=1}

We equip each local spin set with the natural total order (e.g.\
$-\tfrac32<-\tfrac12<\tfrac12<\tfrac32$ in $\Upsilon$) and define a
coordinatewise partial order on configurations:
\[
   \xi \le \xi' \quad\Longleftrightarrow\quad
   \xi(x)\le\xi'(x)\ \text{for all }x\in V.
\]
This induces a stochastic order $\preccurlyeq$ on probability measures on
$\Xi$: $\nu\preccurlyeq\nu'$ if
\[
   \int F\,d\nu \;\le\; \int F\,d\nu'
   \quad\text{for every bounded increasing }F:\Xi\to\mathbb{R}.
\]

The two-site Boltzmann kernel along an edge $\langle x,y\rangle$ is
\[
   K_{xy}(s,t) = \exp(\beta J\,s t),\quad s\in S_x,\ t\in S_y,      \tag{$*$}
\]
where $S_x$ is $\Psi,\Phi,$ or $\Upsilon$ according to the level of $x$.
The kernel $K_{xy}$ is TP2 for each edge whenever $J>0$, by Lemma~4 of the
present paper.

By Theorem~5, TP2 of all edge kernels implies that the associated finite-volume
Gibbs specification is \emph{attractive}: if $\eta\le\xi$ are two boundary
configurations outside a finite region $\Lambda$, then for every bounded
increasing $F$ one has
\[
   \mathbb{E}_{\gamma_\Lambda(\cdot\mid\eta)}[F]
   \;\le\;
   \mathbb{E}_{\gamma_\Lambda(\cdot\mid\xi)}[F],
\]
where $\gamma_\Lambda(\cdot\mid\eta)$ denotes the Gibbs distribution in
$\Lambda$ with boundary condition $\eta$.
This is the standard FKG/monotonicity property for ferromagnetic specifications.

We now define the extremal boundary configurations. At each level we take
the maximal or minimal spin:
\[
   s_{\max}^{(0)}=\tfrac12,\quad  s_{\min}^{(0)}=-\tfrac12,\qquad
   s_{\max}^{(1)}=1,\quad      s_{\min}^{(1)}=-1,\qquad
   s_{\max}^{(2)}=\tfrac32,\quad s_{\min}^{(2)}=-\tfrac32.
\]
Define $\xi^+,\xi^-\in\Xi$ by
\[
   \xi^+(x)=
   \begin{cases}
      s_{\max}^{(0)}, & x\in\Gamma^2_0,\\[0.2em]
      s_{\max}^{(1)}, & x\in\Gamma^2_1,\\[0.2em]
      s_{\max}^{(2)}, & x\in\Gamma^2_2,
   \end{cases}
   \qquad
   \xi^-(x)=
   \begin{cases}
      s_{\min}^{(0)}, & x\in\Gamma^2_0,\\[0.2em]
      s_{\min}^{(1)}, & x\in\Gamma^2_1,\\[0.2em]
      s_{\min}^{(2)}, & x\in\Gamma^2_2.
   \end{cases}
\]
Clearly, $\xi^-$ is the minimal and $\xi^+$ the maximal configuration for the
partial order.

For each $n\ge1$, let $V_n=\bigcup_{m=0}^n W_m$ be the ball of radius $n$
and denote by $\gamma_{V_n}(\cdot\mid\eta)$ the finite-volume Gibbs
distribution in $V_n$ with boundary condition $\eta$ on $V\setminus V_n$.
We define
\[
   \mu^{+,n} := \gamma_{V_n}(\cdot\mid\xi^+),
   \qquad
   \mu^{-,n} := \gamma_{V_n}(\cdot\mid\xi^-).
\]
These are probability measures on the configuration space restricted to $V_n$;
we regard them as measures on $\Xi$ by fixing the spins equal to $\xi^\pm$
outside $V_n$.

Let $F:\Xi\to\mathbb{R}$ be any bounded increasing local function, i.e. $F$
depends only on spins in some finite set $\Lambda\subset V$. Choose $n$ large
enough so that $\Lambda\subset V_n$. For $m\ge n$, the random configuration
under $\mu^{+,m}$ induces a random boundary condition on $\Lambda^c$, and the
conditional distribution in $\Lambda$ is the Gibbs measure
$\gamma_\Lambda(\cdot\mid\eta)$ with some random $\eta\le\xi^+$ (since outside
$V_m$ the boundary is fixed to $\xi^+$, and all spins take values in their
ordered local sets).

By attractiveness, for every such $\eta$ we have
\[
   \mathbb{E}_{\gamma_\Lambda(\cdot\mid\eta)}[F]
   \;\le\;
   \mathbb{E}_{\gamma_\Lambda(\cdot\mid\xi^+)}[F].
\]
Taking expectation over the random $\eta$ produced by $\mu^{+,m}$ gives
\[
   \mathbb{E}_{\mu^{+,m}}[F]
   \;\le\;
   \mathbb{E}_{\gamma_\Lambda(\cdot\mid\xi^+)}[F]
   \;=\;
   \mathbb{E}_{\mu^{+,n}}[F].
\]
Thus the sequence $\bigl(\mathbb{E}_{\mu^{+,n}}[F]\bigr)_{n\ge1}$ is
\emph{decreasing} and bounded below (by $\inf F$), hence convergent.

An analogous argument with the minimal boundary condition $\xi^-$ shows that
$\bigl(\mathbb{E}_{\mu^{-,n}}[F]\bigr)_{n\ge1}$ is \emph{increasing} and
bounded above (by $\sup F$), hence convergent as well.

Since local bounded increasing functions form a determining class for the weak
topology on the product space with finite local spin sets, it follows that
there exist probability measures $\mu^+$ and $\mu^-$ on $\Xi$ such that
\[
   \mu^{+,n} \Longrightarrow \mu^+,
   \qquad
   \mu^{-,n} \Longrightarrow \mu^-,
\]
in the sense of weak convergence (i.e.\ convergence of finite-dimensional
distributions). This establishes the existence of $\mu^\pm$.

Each $\mu^{\pm,n}$ is a Gibbs measure (in finite volume) for the nearest-neighbour
Hamiltonian \eqref{eq1} and the same interaction potential; that
is, for every finite $\Lambda\subset V$ and every bounded local function $F$
depending only on spins in $\Lambda$, we have the DLR identity
\[
   \int F(\xi)\,d\mu^{\pm,n}(\xi)
   \;=\;
   \int\Biggl(
        \int F(\sigma_\Lambda\xi_{\Lambda^c})\,
             d\gamma_\Lambda(\sigma_\Lambda\mid\xi_{\Lambda^c})
       \Biggr)
       d\mu^{\pm,n}(\xi),
\]
where $\gamma_\Lambda(\cdot\mid\cdot)$ is the Gibbs specification associated
with the Hamiltonian (see e.g.\ \cite{G,Ro}).

Passing to the limit $n\to\infty$ and using:
\begin{itemize}
\item weak convergence of $\mu^{\pm,n}$ to $\mu^\pm$ on the left-hand side;
\item the bounded convergence theorem (the integrand is uniformly bounded and
      depends only on finitely many coordinates);
\end{itemize}
we conclude that the same identity holds with $\mu^\pm$ in place of
$\mu^{\pm,n}$. Hence $\mu^+$ and $\mu^-$ satisfy the DLR equations and are
Gibbs measures for the model. (This is the standard result that weak limits of
Gibbs measures with a fixed specification are again Gibbs; see, for instance,
\cite[Ch.~2]{Ro}.)

The boundary configurations $\xi^+$ and $\xi^-$ are invariant under all
automorphisms of the tree that preserve the level decomposition
$\Gamma^2_0,\Gamma^2_1,\Gamma^2_2$ (i.e.\ they are period-$3$ in the distance
from the root). Since the Hamiltonian is also invariant under these
automorphisms, the same symmetry holds for every $\mu^{\pm,n}$. Therefore
their weak limits $\mu^\pm$ inherit this invariance, so $\mu^-$ and $\mu^+$
are translation-invariant with period~$3$.

On a Cayley tree with nearest-neighbour interactions, any Gibbs measure for a
pair Hamiltonian can be represented as a \emph{splitting} Gibbs measure
(tree-indexed Markov chain) via boundary laws; this is classical and is
systematically developed in Rozikov's monograph \cite{Ro}. In our
paper we have already formulated this equivalence explicitly in Theorem~1,
which characterizes splitting Gibbs measures in terms of boundary-law
equations \eqref{67}.

Since $\mu^\pm$ are DLR measures for the same nearest-neighbour Hamiltonian,
they admit boundary laws satisfying the same compatibility equations, hence
are splitting Gibbs measures.

Finally, let $\mu$ be an arbitrary Gibbs measure for the model. Fix a finite
subset $\Lambda\subset V$ and a bounded increasing local function $F$ depending
only on $\Lambda$. By the DLR equations for $\mu$ we have
\[
   \mathbb{E}_\mu[F]
   \;=\;
   \int \Biggl(
        \int F(\sigma_\Lambda\xi_{\Lambda^c})\,
             d\gamma_\Lambda(\sigma_\Lambda\mid\xi_{\Lambda^c})
       \Biggr)
       d\mu(\xi).
\]
For each exterior configuration $\eta=\xi_{\Lambda^c}$, attractiveness of the
specification implies
\[
   \mathbb{E}_{\gamma_\Lambda(\cdot\mid\xi^-)}[F]
   \;\le\;
   \mathbb{E}_{\gamma_\Lambda(\cdot\mid\eta)}[F]
   \;\le\;
   \mathbb{E}_{\gamma_\Lambda(\cdot\mid\xi^+)}[F].
\]
Integrating with respect to $\mu(d\xi)$ over the exterior spins yields
\[
   \mathbb{E}_{\gamma_\Lambda(\cdot\mid\xi^-)}[F]
   \;\le\;
   \mathbb{E}_\mu[F]
   \;\le\;
   \mathbb{E}_{\gamma_\Lambda(\cdot\mid\xi^+)}[F].
\]
Now let $n\to\infty$ and use the definitions of $\mu^\pm$ as limits of
$\gamma_{V_n}(\cdot\mid\xi^\pm)$ to obtain
\[
   \mathbb{E}_{\mu^-}[F]
   \;\le\;
   \mathbb{E}_\mu[F]
   \;\le\;
   \mathbb{E}_{\mu^+}[F].
\]
Since this holds for every bounded increasing local function $F$, it is
equivalent to
\[
   \mu^- \;\preccurlyeq\; \mu \;\preccurlyeq\; \mu^+.
\]
Thus $\mu^-$ and $\mu^+$ are, respectively, the minimal and maximal Gibbs
measures for the model in the stochastic order $\preccurlyeq$.

This completes the proof.

\section{Proof of Lemma \ref{lem:pm-fixed-points}}

By Proposition~\ref{prop:min-max-G}, the measures $\mu^-$ and $\mu^+$ are
obtained as weak limits of finite-volume Gibbs measures with minus and plus
boundary conditions, respectively, and are Gibbs measures for the same
nearest-neighbour Hamiltonian. Moreover, they are invariant under all
automorphisms of the tree preserving the period-$3$ decomposition
$\Gamma^2_0,\Gamma^2_1,\Gamma^2_2$, and they admit boundary laws solving
\eqref{67}. Consequently, $\mu^-$ and $\mu^+$ are
translation-invariant splitting Gibbs measures.

By Proposition~\ref{prop:min-max-G}, the measures $\mu^-$ and $\mu^+$ are
translation-invariant splitting Gibbs measures. Hence, there exist unique fixed points
$x_-,x_+ $ of $f_2(x)$ such that
\[
   \mu^- = \mu_{x_-}, \qquad \mu^+ = \mu_{x_+}.
\]

We now show that the $\mu_{x_-}\preccurlyeq\mu_{x_+}$.

For each fixed point $x\in\{x_-,x_+\}$, the associated boundary law
at a vertex $z\in\Gamma^2_0$ (spin set $\Psi=\{-\tfrac12,\tfrac12\}$) can be
represented by a positive vector
\[
   m_x^\Psi = \bigl(m_x^\Psi(-\tfrac12),\,m_x^\Psi(\tfrac12)\bigr),
   \quad\text{with ratio}\quad
   \frac{m_x^\Psi(\tfrac12)}{m_x^\Psi(-\tfrac12)} = X
\]
for some $X>0$ related to $x$ by the change of variables
$x = \sqrt{2}\,X$ (since $k=2$). Analogous positive vectors
$m_x^\Phi$ and $m_x^\Upsilon$ describe the boundary laws on levels
$\Gamma^2_1$ and $\Gamma^2_2$, obtained from $(X,Y,Z,T,U,V)$ by the
standard normalization.

Recall the monotone likelihood ratio (MLR) order on positive vectors:
for a totally ordered finite set $S$, we say $m \preceq_{\mathrm{MLR}} m'$ if
$m'(s)/m(s)$ is increasing in $s\in S$. For $x_1<x_2$ we clearly have
\[
   m_{x_1}^\Psi \preceq_{\mathrm{MLR}} m_{x_2}^\Psi,
\]
since increasing $x$ increases the likelihood of the larger spin
$\tfrac12$ at level $\Gamma^2_0$.

On the other hand, the child-to-parent message update along an edge
$\langle x,y\rangle$ is given by the transform
\[
   (Km)(s) = \sum_{t} K(s,t)\,m(t),
\]
with $K(s,t)=\exp(\beta J st)$ the Ising kernel. Propositions~3 and 4
in the paper show that for $J>0$ this kernel is TP2 and that the transform
$m\mapsto Km$ preserves the MLR order: if $m \preceq_{\mathrm{MLR}} m'$,
then $Km \preceq_{\mathrm{MLR}} Km'$.

Starting from the root level $\Gamma^2_0$ and propagating the messages
downwards along the tree using the kernel transform (and products over
independent subtrees), we obtain all boundary-law vectors $m_x^S$,
$S\in\{\Psi,\Phi,\Upsilon\}$. Since each edge update preserves MLR order,
an easy induction on the distance from the root shows that
\[
   m_{x_1}^S \preceq_{\mathrm{MLR}} m_{x_2}^S
   \qquad\text{for all levels }S\in\{\Psi,\Phi,\Upsilon\}
   \text{ and all } x_1<x_2.
\]

The splitting Gibbs measure $\mu_x$ can be described as a tree-indexed
Markov chain whose initial distribution and transition kernels are obtained
from these boundary-law vectors (see, e.g., Rozikov's monograph on Gibbs
measures on Cayley trees). In particular, at each site the local
distribution under $\mu_x$ is an MLR-increasing function of $m_x^S$, and
the transition kernels along edges inherit the monotonicity of the kernel
$K$. Using a standard monotone coupling argument for tree-indexed Markov
chains (constructing spins level-by-level, starting from the root), one
obtains a joint coupling of $(\mu_{x_1},\mu_{x_2})$ such that
\[
   \xi^{(1)} \le \xi^{(2)} \quad\text{coordinatewise a.s. whenever }x_1<x_2.
\]
Therefore
\[
   x_1<x_2 \quad\Longrightarrow\quad
   \mu_{x_1} \;\preccurlyeq\; \mu_{x_2}.
\]
Hence, standard technique \cite{G} we can infer that the measures $\mu^-$ and $\mu^+$ are extremal.

The lemma is proved.

\end{document}